\documentclass[12pt]{article}
\usepackage{amsmath,amssymb,amsthm, mathtools}

\newcommand{\D}{\mathrm{D}}

\long\def\symbolfootnote[#1]#2{\begingroup%
\def\thefootnote{\fnsymbol{footnote}}\footnote[#1]{#2}\endgroup}

\newtheorem{thm}{Theorem}[section]
\newtheorem{prop}[thm]{Proposition}
\newtheorem{lem}[thm]{Lemma}

\theoremstyle{definition}

\newtheorem{defn}[thm]{Definition}

\newtheorem{example}[thm]{Example}

\theoremstyle{remark}

\newtheorem{rem}[thm]{Remark}
\usepackage{comment}

\title{Cauchy singular integral operator with parameters in  Log-H\"older spaces }

\author{Yifei Pan and Yuan Zhang\footnote{partially supported by NSF DMS-1501024}}
\date{}

\begin{document}

\maketitle

\begin{abstract}
This paper is motivated by a claim in  the classical textbook of Muskhelishvili concerning  the Cauchy singular integral operator $S$ on H\"older functions with   parameters. To the contrary of  the claim, a counter example   was constructed by Tumanov which shows that $S$ with parameters fails to maintain the same H\"older regularity with respect to the parameters.
In view of the example,   the behavior of the Cauchy singular integral operator with parameters between a type of  Log-H\"older spaces is  investigated to obtain the sharp norm estimates. At the end of the paper, we discuss its application to the $\bar\partial$ problem on product domains.
\end{abstract}

\section{Introductions}

Let $D$  be a bounded domain in $\mathbb C$,  $\Lambda$ be (the closure of) an open set in $\mathbb R$ or $\mathbb C$ and $\Omega: = D\times \Lambda$.  In particular, $\partial D$ consists of a finite number of $C^{1, \alpha}$  Jordan
curves possessing no points in common.  Given a complex-valued function $f\in C^\alpha(\Omega)$, define the Cauchy singular integral  along the slice $D$ as follows. For any $(z, \lambda)\in \Omega$,
\begin{equation}\label{S2}
      S f(z, \lambda):=\frac{1}{2\pi i}\int_{\partial D}\frac{f(\zeta,\lambda)}{\zeta-z}d\zeta.
 \end{equation}
Classical singular integral operators theory in one complex variable states that, there exists a constant $C$ dependent only on $\Omega $ and $ \alpha$, such that
  $Sf(\cdot, \lambda)\in C^{ \alpha}(D)$ for each  $\lambda\in \Lambda$, and $$\|Sf(\cdot, \lambda)\|_{C^{ \alpha}(D)}\le C\|f(\cdot, \lambda)\|_{C^{ \alpha}(D)}.$$ (See for instance \cite{Muskhelishvili}\cite{V} et al.) It is plausible to ask whether $S$ in (\ref{S2}) is a bounded linear operator in $C^{\alpha}(\Omega)$. The question was claimed to be true by Muskhelishvili (see \cite{Muskhelishvili} p. 49-50). In fact,   Muskhelishvili's proof only  shows that given any arbitrarily small $\epsilon$ with $0<\epsilon<\alpha$, $S$ is bounded sending  $C^{ \alpha}(\Omega)$ into $C^{ \alpha-\epsilon}(\Omega)$.

To the contrary of Muskhelishvili's claim, Tumanov \cite{Tumanov} (p. 486) constructed a concrete example showing that $S$ with parameters fails to maintain the same H\"older regularity with respect to the parameters.
In order to  study the optimal parameter dependence of $S$ in (\ref{S2}) on $\lambda$, we introduce the following Log-H\"older spaces,  which are considered as  refined H\"older spaces and would  naturally capture the boundedness of the Cauchy singular integral operator.



\begin{defn}\label{def}
Let $\Omega$  be a  domain in $ \mathbb R^n$, $k\in \mathbb Z^+\cup \{0\}$, $0<\alpha\le 1$ and $\nu\in\mathbb R$. A function $f\in C^k(\Omega)$ is said to be in $ C^{k, L^\alpha Log^\nu L}(\Omega) $ if
$$\|f\|_{C^{k, L^\alpha Log^\nu L} (\Omega)}:  =\sum_{|\gamma|=0}^k\sup_{w\in \Omega}|D^\gamma f(w)|+ \sum_{|\gamma|=k}\sup_{w, w+h\in \Omega, ,\\ 0<|h|\le \frac{1}{2}} \frac{|D^\gamma f(w+h) - D^\gamma f(w)|}{ |h|^\alpha|\ln |h||^{\nu}}<\infty.$$

\end{defn}
Note that when $\alpha=1$ and $\nu<0$, $ C^{k, L^\alpha Log^\nu L}(\Omega) $  consists of constant functions only and thus becomes trivial. Without loss of generality, we always assume $\nu\ge 0$ if $\alpha=1$ in the rest of the paper. It can be verified that $ C^{k, L^\alpha Log^\nu L} (\Omega)$ is a Banach space. Moreover,
  for any $\mu, \nu\in \mathbb R^+, k\in\mathbb Z^+\cup\{0\},  0<\epsilon< \alpha< 1$,  $ C^{k, \alpha+\epsilon}(\Omega)\xhookrightarrow{i}  C^{k, L^\alpha Log^{-\nu-\mu}L} (\Omega) \xhookrightarrow{i}   C^{k, L^\alpha Log^{-\nu}L} (\Omega) \xhookrightarrow{i}  C^{k, \alpha}(\Omega)\xhookrightarrow{i}   C^{k, L^\alpha Log^\nu L} (\Omega)\xhookrightarrow{i}   C^{k, L^\alpha Log^{\nu+\mu}L} (\Omega)\xhookrightarrow{i}  C^{k, \alpha-\epsilon} (\Omega) $, where the  inclusion map $i$ at each level is a continuous embedding.  The Log-H\"older space $ C^{k, L^\alpha Log^\nu L}(\Omega)$ reduces to the well-understood Log-Lipschitz space $ C^{k, L^1LogL}(\Omega)$ when $k=0$ and $\nu=\alpha=1$, and to H\"older space $C^{k, \alpha}(\Omega)$  when $\nu=0$. Our main theorem stated below  shows that $S$ is a bounded operator from $C^{k, L^\alpha Log^\nu L} (\Omega)$ into $C^{k, L^\alpha Log^{\nu+1}  L} (\Omega), k\in \mathbb Z^+\cup \{0\},  0<\alpha\le 1, \nu\in\mathbb R$.


\begin{thm}\label{main}
Let $D$  be a bounded domain in $\mathbb C$ with $C^{k, \alpha}$  boundary,  $k\in \mathbb Z^+\cup \{0\},  0<\alpha\le 1$,  $\Lambda$ be an open set in $\mathbb R$ or $\mathbb C$, and $\Omega: = D\times \Lambda$.   Then $S$ defined in (\ref{S2}) sends $C^{k, L^\alpha Log^\nu L} (\Omega)$ into $C^{k, L^\alpha Log^{\nu+1}  L} (\Omega)$,  $ \nu\in\mathbb R$. Moreover, for any $f\in C^{k, L^\alpha Log^\nu L} (\Omega)$,
  \begin{equation*}
      \|Sf\|_{C^{k, L^\alpha Log^{\nu+1}  L} (\Omega)}\le C\|f\|_{C^{k, L^\alpha Log^\nu L} (\Omega)},
  \end{equation*}
  where $C$ is some constant dependent only on $\Omega, k, \alpha$ and $\nu$.
\end{thm}

In view of Tumanov's example, Theorem \ref{main} is optimal in the sense that the target space $C^{k, L^\alpha Log^{\nu+1}  L} (\Omega)$ can not be replaced by  $C^{k, L^\alpha Log^{\nu+\mu}  L}(\Omega)$ for any $\mu<1$. 
As an application of the theorem, we study solutions  in Log-H\"older spaces to the $\bar\partial$ problem  on product domains, improving the regularity result of \cite{PZ}.

\begin{thm}
 \label{cor}
Let $D_j\subset\mathbb C, j= 1, \ldots, n, $ be bounded domains with  $C^{k+1,\alpha}$ boundary, $ n\ge 2,  k\in \mathbb Z^+\cup \{0\}, 0<\alpha\le 1$, and $\Omega: = D_1\times\cdots\times D_n$. Assume $\mathbf f=\sum_{j=1}^nf_jd\bar z_j\in C^{k, L^\alpha Log^\nu L}(\Omega), \nu\in \mathbb R,$ is a $\bar\partial$-closed (0,1) form on $\Omega$  (in the sense of distributions if $k=0$). There exists a  solution operator $T$ to  $\bar\partial u =\mathbf f$
 such that  $T\mathbf f\in C^ {k, L^\alpha Log^{\nu+n-1}L}(\Omega)$,  $\bar\partial T\mathbf f = \mathbf f $ (in the sense of distributions if $k=0$) and $\|T\mathbf f\|_{C^{k, L^\alpha Log^{\nu+n-1}L}(\Omega)}\le C\|\mathbf f\|_{C^{k, L^\alpha Log^{\nu}L}(\Omega)}$, where $C$ depends only on $\Omega, k, \alpha$ and $\nu$.
\end{thm}
We would like to point out, unlike smooth domains,   there  is no gain of regularity phenomenon for the $\bar\partial$ problem  on product domains, as indicated  by an example of Stein and Kerzman \cite{Kerzman} in $L^\infty$ space (See also \cite{PZ} for examples in H\"older spaces). One can similarly construct examples to show that the $\bar\partial$ problem on product domains does not gain regularity in Log-H\"older spaces. Yet it is  not clear  whether there exists  a solution operator that can achieve the same regularity as that of the data space.

\begin{example}\label{ex2}
Let $\triangle^2=\{(z_1, z_2)\in \mathbb C^2: |z_1|<1, |z_2|<1\}$ be the bidisc. For each $k\in \mathbb Z^+\cup \{0\}$, $ 0<\alpha<1$ and $\nu\in\mathbb R$, consider $\bar\partial u =\mathbf f:= \bar\partial ((z_1-1)^{k+\alpha}\bar z_2\log^\nu (z_1-1))$ on $\triangle^2$, $\frac{1}{2}\pi <\arg (z_1-1)<\frac{3}{2}\pi$. Then $\mathbf f= (z_1-1)^{k+\alpha}\log^\nu (z_1-1)d\bar z_2  \in C^{k,L^\alpha Log^{\nu}L}(\triangle^2)$ is a $\bar\partial$-closed $(0,1)$ form. However, there does not exist a solution $u\in C^{k, L^\beta Log^{\nu}L}(\triangle^2)$ to $\bar\partial u =\mathbf f$ for any $\beta$ with $\beta>\alpha$.
 \end{example}

The rest of the paper is organized as follows. In Section 2, preliminaries about the function spaces and (semi-)norms are defined, as well as the classical theory about the Cauchy type integrals.  The example of Tumanov is discussed in Section 3 to show that $S$ does not send $C^\alpha(\triangle^2)$ into itself, $0<\alpha<1$. Section 4 is devoted to the boundedness of the Cauchy singular integral operator between Log-H\"older spaces on the complex plane. In Section 5 and Section 6,  Theorem \ref{main} and Theorem \ref{cor} are proved respectively, along with the verification of   Example \ref{ex2}. 

\section{Preliminaries and Notations}
Throughout the rest of the paper,    $k, \mu, \nu$ and $ \alpha$  are  always referred to (part of) the indices of the Log-H\"older spaces. $\gamma$ may represent either a positive integer or an $n$-tuple, determined by the context.  $C$ represents a constant that is dependent only on $\Omega, k,  \nu$ and $\alpha$, which may be of different values in different places.

For convenience of notations, given $f\in C^{k, L^\alpha Log^\nu L}(\Omega)$, denote by $$\|f\|_{C^k(\Omega)}: = \sum_{|\gamma|=0}^k\sup_{w\in \Omega}|D^\gamma f(w)|$$ and the  semi-norm $$ H^{\nu}[f]: =\sup_{w, w+h\in \Omega, \\ 0<|h|\le \frac{1}{2}} \frac{| f(w+h) -  f(w)|}{ |h|^\alpha|\ln |h||^{\nu}}.$$
Here  $\alpha$ is suppressed from the above notation due to a fixed value of $\alpha$  throughout the paper. When $\nu=0$, we also suppress $\nu$ and write $H[\cdot]$ for $H^0[\cdot]$.  Consequently, $\|f\|_{C^{k, L^\alpha Log^\nu L}(\Omega)} =\|f\|_{C^k(\Omega)}+ \sum_{|\gamma|=k}H^{\nu} [D^\gamma f]$.

It is worth noting that the upper bound $\frac{1}{2}$  of $|h|$  under the supreme for $H^{\nu}[f]$   is not essential. It can be replaced by any positive number less than 1 without changing the function space $C^{k, L^\alpha Log^\nu L}(\Omega) $, and the resulting norm is equivalent  by some constant  dependent only on  $D$, $\alpha$, $\nu$ and the positive number itself.

 In particular when $\Omega=D\times \Lambda$,  the H\"older semi-norms along $z$ and $\lambda$ variables for each fixed $\lambda\in \Lambda$ and fixed $z\in D$ respectively can be defined as follows.
 \begin{equation*}\begin{split}
   &H^{\nu}_D[f(\cdot,\lambda)]:=\sup_{\zeta, \zeta+h \in D, 0<|h|\le \frac{1}{2}} \frac{|f(\zeta+h, \lambda) - f(\zeta, \lambda)|}{|h|^\alpha |\ln |h||^{\nu}};\\
   &H^{\nu}_\Lambda[f(z,\cdot)]:=\sup_{\zeta, \zeta+h \in \Lambda,  0<|h|\le \frac{1}{2}} \frac{|f(z, \zeta+h) - f(z, \zeta)|}{|h|^\alpha |\ln |h||^{\nu}}.
   \end{split}
 \end{equation*}
The above two expressions are clearly bounded by $H^{\nu}[f]$ by definition. On the other hand, the following elementary property for Log-H\"older semi-norms can be observed.
\medskip

\begin{lem}\label{comp}
There exists a constant $C$ dependent only on $\Omega, \alpha$ and $\nu$, such that for any function $f
\in C^{L^\alpha Log^\nu L}(\Omega), $ \begin{equation*}
    \|f\|_{C^{ L^\alpha Log^\nu L}(\Omega)} \le C(\|f\|_{C(\Omega)}+\sup_{\lambda\in \Lambda}H^{\nu}_D[f(\cdot,\lambda)] + \sup_{z\in D}H^{\nu}_\Lambda[f(z, \cdot)]).
\end{equation*} \end{lem}

\begin{proof}
 We only need to show $H^{\nu} [ f]\le C(\|f\|_{C(\Omega)}+\sup_{\lambda\in \Lambda}H^{\nu}_D[f(\cdot,\lambda)] + \sup_{z\in D}H^{\nu}_\Lambda[f(z, \cdot)])$. Indeed, for any $w=(z, \lambda)\in D\times \Lambda, w+h=(z+h_1, \lambda+h_2)\in D\times \Lambda$ with $|h|\le r_0: = \min\{e^{-\frac{\nu}{\alpha}}, \frac{1}{2}\}$, then $(z+h_1, \lambda)\in D\times \Lambda$. Hence \begin{equation*}
     \begin{split}
         |f(w+h)- f(w)|&\le | f(z+h_1, \lambda+h_2)-f(z+h_1, \lambda)|+| f(z+h_1, \lambda) - f(z, \lambda)|\\
         &\le  |h_2|^\alpha|\ln |h_2||^{\nu}\sup_{z\in D}H^{\nu}_\Lambda[f(z, \cdot)] + |h_1|^\alpha|\ln |h_1||^{\nu}\sup_{\lambda\in \Lambda}H^{\nu}_D[f(\cdot,\lambda)] \\
         &\le  |h|^\alpha|\ln |h||^{\nu}(\sup_{\lambda\in \Lambda}H^{\nu}_D[f(\cdot,\lambda)] + \sup_{z\in D}H^{\nu}_\Lambda[f(z, \cdot)]).
     \end{split}
 \end{equation*}
   Here the last inequality is due to the non-decreasing property of the real-valued function $s^\alpha|\ln s|^\nu$ on the interval $(0, r_0)$.
\end{proof}

Let $D$ be a bounded domain in $\mathbb C$ with $C^{k+1,\alpha}$  boundary, $k\in \mathbb Z^+\cup \{0\}, 0<\alpha\le 1$. Given a complex valued  function $f\in C(\bar D)$,  the following two operators related to the Cauchy kernel are well defined for $z\in D$.

\begin{equation}\label{TS1}
\begin{split}
 Tf(z): &=\frac{-1}{2\pi i}\int_\D \frac{f(\zeta)}{\zeta- z}d\bar{\zeta}\wedge d\zeta;\\
Sf(z): &=\frac{1}{2\pi i}\int_{\partial D}\frac{f(\zeta)}{\zeta- z}d\zeta.
\end{split}
\end{equation}
Here  the positive orientation of $\partial D$ is  such  that the domain $D$ is always to its left while traversing along the contour(s). We state  some classical results concerning the Cauchy type integrals $T$ and $S$ on the complex plane.  The reader may check for instance  \cite{V}  for reference.

\begin{thm}
\label{123}
  Let D be a bounded domain with $C^{k+1, \alpha}$ boundary.\\
 1) If $f\in L^p(D), p>2$, then $Tf\in C^\alpha(D),  \alpha=\frac{p-2}{p}$. Moreover, $$\|Tf\|_{C^{\alpha}(D)}\le C\|f\|_{L^p},$$
 for some constant $C$ dependent only on $D$ and $p$.\\  
 2) If $f\in C^{k, \alpha}(D), k\in \mathbb Z^+\cup \{0\}, 0<\alpha<1$. Then $Tf\in C^{k+1, \alpha}(D)$ and $Sf\in C^{k, \alpha}(D)$. Moreover,
  \begin{equation*}
    \begin{split}
      &\|Tf\|_{C^{k+1, \alpha}(D)}\le C\|f\|_{C^{k, \alpha}(D)};\\
      &\|Sf\|_{C^{k, \alpha}(D)}\le C\|f\|_{C^{k, \alpha}(D)}
    \end{split}
  \end{equation*}
  for some constant $C$ dependent only on $D, k$ and $\alpha$.
  \end{thm}
\medskip


\section{$S$ does not sent $C^{\alpha}(\triangle^2)$ into itself}
In this section, we verify in detail  Tumanov's  example  in \cite{Tumanov}  (See also \cite{MP}) that $S$ defined in (\ref{S2}) does not send $C^{\alpha}(\triangle^2)$ into itself, $0<\alpha<1$.
Define for $\lambda\in \triangle$,
$$
  \tilde f(e^{i\theta}, \lambda) =
  \left\{
      \begin{array}{cc}
     |\lambda|^\alpha, & -\pi\le \theta \le -|\lambda|^\frac{1}{2}; \\
     \theta^{2\alpha}, & -|\lambda|^\frac{1}{2}\le \theta\le 0; \\
      \theta^\alpha, & 0\le \theta\le |\lambda|;\\
      |\lambda|^\alpha, &|\lambda|\le \theta\le \pi.
    \end{array}
\right.$$
Then $\tilde f\in C^\alpha(\partial \triangle \times \triangle)$. Extend $\tilde f$ onto $\triangle^2$,  denoted as $f$, such that $f\in C^\alpha(\triangle^2) $ and $\|f\|_{C^\alpha(\triangle^2)} = \|\tilde f\|_{C^\alpha(\partial \triangle \times \triangle)}$. (For instance, for each $w\in \triangle^2$, let $f(w): = \inf_{\eta\in \partial \triangle \times \triangle}\{\tilde f(\eta)+ M|w-\eta|^\alpha\}$, where $M = \|\tilde f\|_{C^\alpha(\partial \triangle \times \triangle)}$.)

We first show that  $Sf(1, \cdot)\notin C^\alpha(\triangle)$. Indeed, a direct computation gives for  $\lambda\in \triangle$,
\begin{equation*}
    \begin{split}
  2\pi i Sf(1, \lambda)
  = &\int_{\partial\triangle} \frac{\tilde f(\zeta, \lambda)}{\zeta-1}d\zeta\\
   = &i\int_{-\pi}^\pi \frac{\tilde f(e^{i\theta},\lambda)e^{i\theta} }{e^{i\theta}-1}d\theta   \\
   = & \frac{1}{2 }\int_{-\pi}^\pi \frac{\tilde f(e^{i\theta},\lambda)e^{\frac{i\theta}{2}} }{\sin\frac{\theta}{2}}d\theta\\
   = &\frac{1}{2 }\int_{-\pi}^\pi \tilde f(e^{i\theta},\lambda)\cot\frac{\theta}{2}d\theta  +\frac{i}{2 }\int_{-\pi}^\pi \tilde f(e^{i\theta},\lambda)d\theta
   =: I+II.
    \end{split}
\end{equation*}
Here the third equality uses the identity that $e^{i\theta}-1 = \cos\theta-1 +i\sin\theta =  2i\sin\frac{\theta}{2}e^{\frac{i\theta}{2}}$. Since $\tilde f\in C^\alpha(\partial \triangle \times \triangle)$, we have $II\in C^\alpha(\triangle)$.

On the other hand, write $$I =\frac{1}{2}\int_{-\pi}^\pi \tilde f(e^{i\theta},\lambda)(\cot\frac{\theta}{2} -\frac{2}{\theta} )d\theta +   \int_{-\pi}^{\pi} \frac{\tilde f(e^{i\theta},\lambda)}{\theta}d\theta. $$
Notice that $ \cot \frac{\theta}{2} -\frac{2}{\theta}$ extends as a continuous function on $[-\pi, \pi]$. Hence $\int_{-\pi}^\pi \tilde f(e^{i\theta},\lambda)(\cot\frac{\theta}{2} -\frac{2}{\theta} )d\theta \in  C^\alpha(\triangle)$ as a function of $\lambda\in \triangle$. For the second term in $I$,  from construction of $\tilde f$,
\begin{equation*}
    \begin{split}
 \int_{-\pi}^{\pi} \frac{\tilde f(e^{i\theta},\lambda)}{\theta}d\theta =&  \int_{-|\lambda|^\frac{1}{2}}^0 \frac{\theta^{2\alpha}}{\theta} d\theta +\int_0^{|\lambda|}\frac{\theta^\alpha}{\theta}d\theta +  \int_{|\lambda|}^{|\lambda|^\frac{1}{2}}  \frac{|\lambda|^\alpha}{\theta}d\theta \\
 = & \frac{|\lambda|^\alpha}{2\alpha} +\frac{1}{2}|\lambda|^\alpha|\ln|\lambda||.
    \end{split}
\end{equation*}
 We thus obtain $I\notin C^\alpha(\triangle)$ and hence $Sf(1, \cdot)\notin C^\alpha(\triangle)$.

Suppose by contradiction that  $Sf\in C^\alpha(\triangle^2)$. Then the non-tangential limit of $Sf$  on $\partial \triangle \times \triangle$, denoted by $\Phi f$, is in $C^\alpha$ as well. In particular, $\Phi f(1, \cdot)\in C^\alpha(\triangle)$. On the other hand,  by  Sokhotski-Plemelj formula, $\Phi f(1, \cdot) = Sf(1,\cdot) + \frac{1}{2}f(1, \cdot)$. This contradicts with the fact that $Sf(1, \cdot)\notin C^\alpha(\triangle)$.

\begin{rem}\label{rem}
For $f$ constructed above,   $Sf \notin C^{L^\alpha Log^{\mu}  L}(\triangle^2)$ for any $\mu<1$.
\end{rem}



\section{Cauchy singular integral in Log-H\"older spaces in $\mathbb C$}

Let $D$ be a bounded domain in $\mathbb C$ with $C^{1,\alpha}$  boundary, $k\in \mathbb Z^+\cup \{0\}, 0<\alpha\le 1$. In this section, we shall prove that $S$ defined in (\ref{TS1}) is a bounded linear operator from $C^{L^\alpha Log^\nu L}(D)$ into itself if $0<\alpha<1$, and into $C^{L^1 Log^{\nu+1}  L}(D)$ if $\alpha=1$ (and $\nu\ge 0$).   Since $C^{L^\alpha Log^\nu L}(D)$ is a subspace of $C^{\epsilon}(D)$ for $0<\epsilon<\alpha$, $Sf$ is well defined for $f\in C^{L^\alpha Log^\nu L}(D)$ by the classical theory of $S$ in H\"older spaces.

Write $\partial D = \cup_{j=1}^N \Gamma_j$, where each  Jordan curve $\Gamma_j$  is connected and positively oriented  with respect to $D$, and  of total arclength $s_j$. 
 Since $\partial D$ is Lipschitz in particular, $\partial D$ satisfies the so-called {\em chord-arc} condition. In other words, for any $t, t'\in \Gamma_j, j=1, \ldots, N$, let $|t, t'|$ be the smaller length of the two arcs of $\Gamma_j$ with $t$ and $t'$ as the two end points. There exists a  constant $c_0\ge 1$  dependent only on $\partial D$ such that
\begin{equation}\label{arc}
    |t-t'|\le |t, t'|\le c_0|t-t'|.
\end{equation}

The following calculus lemma is elementary but will be frequently used in this section.
\begin{lem}\label{elem}
Let $0<\alpha\le 1$ and $\nu\in \mathbb R$. There exists a constant $C$ dependent only on $\alpha$ and $\nu$, such that for all $0<h\le h_0:=\min\{e^{-\frac{2\nu}{\alpha}}, e^\frac{2\nu}{1-\alpha}, \frac{1}{2}\}$,\\
1) $\int_0^h s^{\alpha-1}|\ln s|^{\nu} ds\le Ch^\alpha|\ln h|^{\nu}$ when $0<\alpha\le1$.\\
2) $\int_h^{h_0} s^{\alpha-2}|\ln s|^{\nu} ds\le\left\{
      \begin{array}{cc}
    Ch^{\alpha-1}|\ln h|^{\nu}, & 0<\alpha<1; \\
     C|\ln h|^{\nu+1}, & \alpha =1.
    \end{array}
\right. $
\end{lem}

\begin{proof}
1) Using  integration by part directly, $$ \int_0^h s^{\alpha-1}|\ln s|^{\nu} ds= \frac{1}{\alpha}\int_0^h |\ln s|^{\nu} ds^\alpha = \frac{1}{\alpha}h^\alpha|\ln h|^{\nu}  + \frac{\nu}{\alpha}\int_0^h s^{\alpha-1}|\ln s|^{\nu -1} ds.$$
If $\nu \le 0$, the lemma follows directly from  the above identity by dropping off the last negative term. If $\nu > 0$, since $s\le h_0\le  e^\frac{-2\nu }{\alpha}$,  $1-\frac{\nu }{\alpha|\ln s|}\ge \frac{1}{2}$, which implies $ \int_0^h s^{\alpha-1}|\ln s|^{\nu} ds - \frac{\nu }{\alpha}\int_0^h s^{\alpha-1}|\ln s|^{\nu -1} ds =  \int_0^h s^{\alpha-1}|\ln s|^{\nu}(1-\frac{\nu}{\alpha|\ln s|}) ds \ge \frac{1}{2} \int_0^h s^{\alpha-1}|\ln s|^{\nu} ds$. Hence
$$ \int_0^h s^{\alpha-1}|\ln s|^{\nu} ds\le  \frac{2}{\alpha}h^\alpha|\ln h|^{\nu}.$$

2) When $0<\alpha<1$, $$\int_h^{h_0} s^{\alpha-2}|\ln s|^{\nu} ds =  \frac{1}{1-\alpha}( h^{\alpha-1}|\ln h|^{\nu} -h_0^{\alpha-1}|\ln h_0|^{\nu})-\frac{\nu}{1-\alpha} \int_h^{h_0} s^{\alpha-2}|\ln s|^{\nu-1} ds.$$
So we have
$$\int_h^{h_0} s^{\alpha-2}|\ln s|^{\nu} ds \le  \frac{1}{1-\alpha}h^{\alpha-1}|\ln h|^{\nu} -\frac{\nu }{1-\alpha} \int_h^{h_0} s^{\alpha-2}|\ln s|^{\nu -1} ds.$$
If $\nu \ge 0$, the lemma is proved as in 1). If $\nu <0$, notice $1+\frac{\nu }{(1-\alpha)|\ln s|}\ge \frac{1}{2}$ when $s\le h_0\le e^\frac{2\nu }{1-\alpha}$, we have $\int_h^{h_0} s^{\alpha-2}|\ln s|^{\nu} ds + \frac{\nu }{1-\alpha} \int_h^{h_0} s^{\alpha-2}|\ln s|^{\nu -1} ds = \int_h^{h_0} s^{\alpha-2}|\ln s|^{\nu} (1+\frac{\nu }{(1-\alpha)|\ln s|}) ds\ge \frac{1}{2}\int_h^{h_0} s^{\alpha-2}|\ln s|^{\nu} ds. $ Hence $\int_h^{h_0} s^{\alpha-2}|\ln s|^{\nu} ds\le \frac{2}{1-\alpha}h^{\alpha-1}|\ln h|^{\nu}$.

When $\alpha =1$ and $\nu \ge 0$, 
$$\int_h^{h_0} \frac{|\ln s|^{\nu}}{s} ds = \frac{1}{\nu +1}(|\ln h|^{\nu+1}-|\ln h_0|^{\nu+1} )\le  \frac{1}{\nu +1}|\ln h|^{\nu+1}.$$Both desired inequalities are proved.
\end{proof}

We first consider points on $\partial D$. When $t\in\partial D$, by Sokhotski-Plemelj Formula (see \cite{Muskhelishvili} for instance), the nontangential limit of $Sf$ at $t\in\partial D$ is
$$\Phi  f(t): = Sf(t)+\frac{1}{2} f(t): =\frac{1}{2\pi i}\int_{\partial D}\frac{f(\zeta)}{\zeta-t}d\zeta + \frac{1}{2} f(t)= \frac{1}{2\pi i}\int_{\partial D}\frac{f(\zeta)-f(t)}{\zeta-t}d\zeta + f(t).$$
Here $Sf(t)= \frac{1}{2\pi i}\int_{\partial D}\frac{f(\zeta)}{\zeta-t}d\zeta$ is interpreted as the Principal Value when $t\in\partial D$ and  is well defined if $f$ is in H\"older spaces. In particular,  $\frac{1}{2\pi i}\int_{\partial D}\frac{1}{\zeta-t}d\zeta = \frac{1}{2}$ when $t\in \partial D$.  Let $h_0$ and $c_0$ be  defined as in Lemma \ref{elem} and (\ref{arc}) respectively,  $s_0: =\min_{1\le j\le N}\{s_j\} >0$ and $\delta_0: = \inf_{1\le j\ne m\le N}\{|t-t'|: t\in \Gamma_j,  t'\in \Gamma_m\}>0$.

\begin{lem}\label{int}
Let $0< \alpha\le 1$. If $f\in C^{L^\alpha Log^\nu L}(D)$, then for $t, t+h\in \partial D$ with $|h|\le \min\{\frac{h_0}{3c_0}, \frac{s_0}{6c_0}, \frac{\delta_0}{2} \}$,
\begin{equation*}
|\Phi f(t+h)-\Phi f(t)|\le \left\{
      \begin{array}{cc}
     C\|f\|_{C^{L^\alpha Log^\nu L}(D)}|h|^\alpha|\ln| h||^{\nu}, & 0<\alpha<1; \\
     C\|f\|_{C^{L^1 Log^\nu L}(D)}|\ln| h||^{\nu+1}, & \alpha =1
    \end{array}
\right.
\end{equation*}
for a constant $C$ dependent only on $D, \alpha$ and $\nu$.
\end{lem}

\begin{proof}
Assume $t\in \Gamma_1$ without loss of generality. Since  $|t+h-t| =|h|\le \frac{\delta_0}{2}$, $ t+h\in \Gamma_1$ as well.    By Sokhotski-Plemelj Formula,
 \begin{equation*}
   \begin{split}
    \Phi f(t+h )-\Phi f(t ) = & \frac{1}{2\pi i}\int_{\partial D}\frac{f(\zeta )-f(t+h )}{\zeta-t-h}d\zeta-\frac{1}{2\pi i}\int_{\partial D}\frac{f(\zeta )-f(t )}{\zeta-t}d\zeta + (f(t+h )-f(t ))\\
    =& \frac{1}{2\pi i}(\int_{\cup_{j=1}^N\Gamma_j}\frac{f(\zeta )-f(t+h )}{\zeta-t-h}d\zeta-\int_{\cup_{j=1}^N\Gamma_j}\frac{f(\zeta )-f(t )}{\zeta-t}d\zeta) + (f(t+h )-f(t )).
   \end{split}
 \end{equation*}

  Because $\cup_{j=2}^N \Gamma_j$ does not intersect with $\Gamma_1$ and $t, t+h\in \Gamma_1$, we have $|\zeta-t|\ge C$ and  $|\zeta-t-h|\ge C$ on $\cup_{j=2}^N \Gamma_j$ for some positive $C$ dependent only on $\partial D$. 
  It immediately follows that
  \begin{equation*}
  \begin{split}
         & \left |\int_{\cup_{j=2}^N\Gamma_j}\frac{f(\zeta )-f(t+h )}{\zeta-t-h}d\zeta-\int_{\cup_{j=2}^N\Gamma_j}\frac{f(\zeta )-f(t )}{\zeta-t}d\zeta\right |\\
          =& \left |\int_{\cup_{j=2}^N \Gamma_j}\frac{(f(\zeta)-f(t))h + (f(t)-f(t+h))(\zeta-t)}{(\zeta-t-h)(\zeta-t)}d\zeta \right |\\
          \le&  \int_{\cup_{j=2}^N\Gamma_j}C\|f\|_{C^{L^\alpha Log^\nu L}(D)}|h|^\alpha|\ln |h||^\nu  |d\zeta|\\
          \le & C\|f\|_{C^{L^\alpha Log^\nu L}(D)}| h|^\alpha|\ln|  h||^{\nu}.
     \end{split}
  \end{equation*}
  It thus suffices to show, in view of the chord-arc condition,  for  $t, t+h\in \Gamma_1$ with $\tilde h: =|t+h, t|\le \min\{\frac{h_0}{3}, \frac{s_0}{6}\}$,
\begin{equation*}
  \left |\int_{\Gamma_1}\frac{f(\zeta )-f(t+h )}{\zeta-t-h}d\zeta-\int_{\Gamma_1}\frac{f(\zeta )-f(t )}{\zeta-t}d\zeta \right | \le \left\{
      \begin{array}{cc}
     C\|f\|_{C^{L^\alpha Log^\nu L}(D)}\tilde h^\alpha|\ln \tilde h|^{\nu}, & 0<\alpha<1; \\
     C\|f\|_{C^{L^1 Log^\nu L}(D)}\tilde h|\ln \tilde h|^{\nu+1}, & \alpha =1.
    \end{array}
\right.
\end{equation*}

Due to the  $C^{1, \alpha}$ boundary of $\Gamma_1$, $|d\zeta|\approx |ds|$. Denote  by s the arclength parameter of $\Gamma_1$ with $\zeta|_{s=0}= t$, and by $l$   the arc on $\Gamma_1$ centered at $t$  of total arclength $4\tilde h$. Recall that $s_1$ is the total arclength of $\Gamma_1$. The chord-arc condition implies $|\zeta-t|\approx |\zeta, t|= \min\{s, s_1-s\}$ on $\Gamma_1$.

On $l$,  notice that $$|\zeta-t-h|\ge C|\zeta, t+h|\ge C||\zeta, t|-|t+h, t|| = \left\{
      \begin{array}{cc}
     C|s-\tilde h|, & s\le \frac{s_1}{2}; \\
    C|s_1-s-\tilde h|, & s\ge \frac{s_1}{2}.
    \end{array}
\right.$$    Together with the fact that  $|f(\zeta )-f(t+h )|\le \|f\|_{C^{L^\alpha Log^\nu L}(D)}|\zeta-t-h|^\alpha |\ln|\zeta-t-h||^{\nu}$ and $|f(\zeta )-f(t )|\le \|f\|_{C^{L^\alpha Log^\nu L}(D)}|\zeta-t|^\alpha |\ln|\zeta-t||^{\nu}$ on $l$, one obtains from Lemma \ref{elem},
\begin{equation*}
 \begin{split}
    &\left |\int_{l}\frac{f(\zeta )-f(t+h )}{\zeta-t-h}d\zeta-\int_{l}\frac{f(\zeta )-f(t )}{\zeta-t}d\zeta \right |\\
    \le &C\|f\|_{C^{L^\alpha Log^\nu L}(D)}(\int_{l}|\zeta-t-h|^{\alpha-1} |\ln|\zeta-t-h||^{\nu} |d\zeta| + \int_{l}|\zeta-t|^{\alpha-1} |\ln|\zeta-t||^{\nu} |d\zeta|)\\
    \le &C\|f\|_{C^{L^\alpha Log^\nu L}(D)}(\int_{0}^{2\tilde h}|s-\tilde h|^{\alpha-1} |\ln |s-\tilde h||^{\nu} ds + \int_{0}^{2\tilde h}|s|^{\alpha-1} |\ln s|^{\nu} ds)\\
     \le &C\|f\|_{C^{L^\alpha Log^\nu L}(D)}(\int_{0}^{3\tilde h}s^{\alpha-1} |\ln s|^{\nu} ds + \int_{0}^{2\tilde h}|s|^{\alpha-1} |\ln s|^{\nu} ds)\\
    \le &C\|f\|_{C^{L^\alpha Log^\nu L}(D)}\tilde h^\alpha|\ln\tilde h|^{\nu}.
 \end{split}
\end{equation*}

Next we estimate
\begin{equation*}
\begin{split}
  &\left |\int_{\Gamma_1\setminus l}\frac{f(\zeta )-f(t+h )}{\zeta-t-h}d\zeta-\int_{\Gamma_1\setminus l}\frac{f(\zeta )-f(t )}{\zeta-t}d\zeta \right |\\
  \le & \left |\int_{\Gamma_1 \setminus l}(f(\zeta )-f(t+h ))(\frac{1}{\zeta-t-h}-   \frac{1}{\zeta-t})d\zeta \right |+\left |\int_{\Gamma_1\setminus l}\frac{f(t+h )-f(t )}{\zeta-t}d\zeta\right|  =: I+II.
  \end{split}
\end{equation*}
Since $II = |f(t+h )-f(t )||\frac{1}{2\pi i}\int_{\Gamma_1-l}\frac{1}{\zeta-t}d\zeta|\le C|f(t+h )-f(t )|$, $II$ is  bounded by $C\|f\|_{C^{L^\alpha Log^\nu L}(D)} \tilde h^\alpha|\ln \tilde h|^{\nu}$. Now we treat  $I =  |\frac{h}{2\pi }\int_{\Gamma_1 \setminus l}\frac{f(\zeta )-f(t+h )}{(\zeta-t-h)(\zeta-t)}d\zeta|$. Due to the chord-arc condition,   $|\zeta, t+h|\ge |\zeta, t|-|t, t+h|= \min \{s -\tilde h, s_1-s-\tilde h\}\ge \tilde h $ on $\Gamma_1\setminus l$. Hence $$|\zeta-t|\le |\zeta, t|\le |\zeta, t+h|+|t+h, t| =|\zeta, t+h|+ \tilde h\le  2 |\zeta, t+h| \le C|\zeta-t-h|,$$ or equivalently, $$|\zeta-t-h|>C |\zeta -t|\approx \min\{s, s_1-s\} $$ on $\Gamma_1\setminus l$. Let $l'$ be the arc on $\Gamma_1$ centered at $t$ with total arclength $\min\{2h_0, s_1\}$ so $l\subset l'\subset \Gamma_1$. Therefore
\begin{equation*}
  \begin{split}
    I \le & C\|f\|_{C^{L^\alpha Log^\nu L}(D)}\tilde h\int_{l'\setminus l}\frac{|\zeta-t-h|^{\alpha-1} |\ln|\zeta-t-h||^{\nu}}{|\zeta-t|}|d\zeta| +\\
    &+C\|f\|_{C(D)}\tilde h\int_{\Gamma_1\setminus l'}\frac{1}{|\zeta-t-h||\zeta-t|}|d\zeta| \\
    \le& C\|f\|_{C^{L^\alpha Log^\nu L}(D)}\tilde h\int_{2\tilde h}^{\min\{h_0, \frac{s_1}{2}\}} s^{\alpha-2}|\ln s|^\nu ds +C\|f\|_{C(D)}\tilde h\int_{\min\{h_0, \frac{s_1}{2}\}}^{\frac{s_1}{2}}\frac{1}{s^2}ds\\
    \le&C\|f\|_{C^{L^\alpha Log^\nu L}(D)}\tilde h(\int_{2\tilde h}^{h_0} s^{\alpha-2}|\ln s|^\nu ds +1).
          \end{split}
\end{equation*}
It follows immediately from Lemma \ref{elem},
\begin{equation*}
\begin{split}
  I \le  \left\{
      \begin{array}{cc}
     C\|f\|_{C^{L^\alpha Log^\nu L}(D)}\tilde h^\alpha|\ln\tilde  h|^{\nu}, & 0<\alpha<1; \\
     C\|f\|_{C^{L^1 Log^\nu L}(D)}\tilde h |\ln \tilde h|^{\nu+1}, & \alpha =1.
    \end{array}
\right.
  \end{split}
\end{equation*}
\end{proof}

For H\"older semi-norm of $S$ at interior points of the domain, classical singular integral operators theory utilizes a generalized version of the Maximum Modulus Theorem of holomorphic functions to a branch of  $\frac{Sf(z)-Sf(z')}{(z-z')^\alpha}$ to achieve the boundedness. 
 We adopt here a different approach introduced in \cite{Mclean}.

Given $t\in\partial D$,  define $\mathcal N(t)$, a nontangential approach region (cf. \cite{Kenig} \cite{Mclean})  as follows.
$$\mathcal N(t)=\{z\in D: |z-t|\le \min \{4 \text{dist}(z, \partial D), \frac{\delta_0}{4}\}\}.$$
If $z\in \mathcal N(t)$, then   $|\zeta-z|\ge \text{dist}(z, \partial D) \ge \frac{1}{4}|z-t|$ for all  $\zeta\in \partial D$. Hence $|\zeta-z|\ge  \frac{1}{4}(|\zeta-t|- |\zeta-z|),$ implying  $|\zeta-z|\ge \frac{1}{5}|\zeta-t|$ on $\partial D$. Altogether, for $z\in \mathcal N(t)$ and $\zeta\in \partial D$,
\begin{equation}\label{444}
   |\zeta-z|\ge\max\{\frac{1}{4}|z-t|, \frac{1}{5}|\zeta-t|\}.
\end{equation}

\begin{lem}\label{bound}  Let $0< \alpha\le 1$. If $f\in C^{L^\alpha Log^\nu L}(D)$ and $t\in\partial D$, then for $z\in \mathcal N(t)$ with $|z-t|\le \min\{h_0, \frac{s_0}{2}\}$,
\begin{equation*}
|Sf(z)-\Phi f(t)|\le \left\{
      \begin{array}{cc}
     C\|f\|_{C^{L^\alpha Log^\nu L}(D)}|z-t|^\alpha|\ln| z-t||^{\nu}, & 0<\alpha<1; \\
     C\|f\|_{C^{L^1 Log^\nu L}(D)}|z-t||\ln| z-t||^{\nu+1}, & \alpha =1
    \end{array}
\right. \end{equation*}
for a constant $C$ dependent only on $D, \alpha$ and $\nu$.
\end{lem}

\begin{proof}
 Without loss of generality, assume $t\in \Gamma_1$. By Cauchy's integral formula, $\frac{1}{2\pi i}\int_{\partial D}\frac{1}{\zeta-z}d\zeta = 1$ when $z\in D$. Hence
  \begin{equation*}
      \begin{split}
   Sf(z)-\Phi f(t)  = & (\frac{1}{2\pi i}\int_{\partial D}\frac{f(\zeta)-f(t)}{\zeta-z}d\zeta + f(t)) - (\frac{1}{2\pi i}\int_{\partial D}\frac{f(\zeta)-f(t)}{\zeta-t}d\zeta + f(t))\\
   =& \frac{z-t}{2\pi i}\int_{\partial D}\frac{f(\zeta)-f(t)}{(\zeta-z)(\zeta-t)}d\zeta\\
   =&\frac{z-t}{2\pi i}\int_{l}\frac{f(\zeta)-f(t)}{(\zeta-z)(\zeta-t)}d\zeta +  \frac{z-t}{2\pi i}\int_{\Gamma_1\setminus l}\frac{f(\zeta)-f(t)}{(\zeta-z)(\zeta-t)}d\zeta +\\
   &+ \frac{z-t}{2\pi i}\int_{\cup_{j=2}^N \Gamma_j}\frac{f(\zeta)-f(t)}{(\zeta-z)(\zeta-t)}d\zeta\\
   =&: I+II+III
   \end{split}
  \end{equation*}
  Here $l$ is the arc on $\Gamma_1$ centered at $t$  of total arclength $2|z-t|=:2|h|$. For  $III$, when $\zeta\in \cup_{j=2}^N \Gamma_j$, $|\zeta-t|\ge \delta_0$, and $|\zeta-z|\ge |\zeta-t|-|t-z|\ge \delta_0-\frac{\delta_0}{4} = \frac{3\delta_0}{4}$. We thus deduce
  \begin{equation*}
      |III|\le C|h|\|f\|_{C(D)}\le C\|f\|_{C^{L^\alpha Log^\nu L}(D)}|h|^\alpha|\ln| h||^{\nu}.
  \end{equation*}
  Next we estimate $I$ and $II$. It follows from (\ref{444}) and Lemma \ref{elem}  that
  \begin{equation*}
      \begin{split}
          |I|\le& C|h|\|f\|_{C^{L^\alpha Log^\nu L}(D)} \int_l\frac{|\zeta-t|^{\alpha-1}|\ln|\zeta-t||^{\nu}}{|\zeta-z|}|d\zeta|\\
          \le& C|h|\|f\|_{C^{L^\alpha Log^\nu L}(D)} \int_l\frac{|\zeta-t|^{\alpha-1}|\ln|\zeta-t||^{\nu}}{|z-t|}|d\zeta|\\
          \le & C\|f\|_{C^{L^\alpha Log^\nu L}(D)} \int_{0}^{|h|}s^{\alpha-1}|\ln s|^{\nu} ds\\
          \le &C\|f\|_{C^{L^\alpha Log^\nu L}(D)}|h|^\alpha|\ln |h||^{\nu}.
      \end{split}
  \end{equation*}
For $II$, let $l'$ be the arc on $\Gamma_1$ centered at $t$  of arclength  $\min\{2h_0, s_1\}$ as in the previous lemma.
  \begin{equation*}
      \begin{split}
      |II|\le & C|h|\|f\|_{C^{L^\alpha Log^\nu L}(D)} \int_{l'\setminus l}\frac{|\zeta-t|^{\alpha}|\ln|\zeta-t||^{\nu}}{|\zeta-t|^2}|d\zeta| + C|h|\|f\|_{C(D)}\int_{\Gamma_1\setminus l'}\frac{1}{|\zeta-t|^{2}}|d\zeta|)\\
      \le & C|h|\|f\|_{C^{ L^\alpha Log^\nu L}(D)} (\int_{|h|}^{h_0} s^{\alpha-2}|\ln s||^{\nu}ds +\int_{\min\{h_0, \frac{s_1}{2}\}}^{\frac{s_1}{2}} \frac{1}{s^2}ds )\\
      \le & \left\{
      \begin{array}{cc}
     C\|f\|_{C^{L^\alpha Log^\nu L}(D)}|h|^\alpha|\ln| h||^{\nu}, & 0<\alpha<1; \\
     C\|f\|_{C^{L^1 Log^\nu L}(D)}|h||\ln| h||^{\nu+1}, & \alpha =1.
    \end{array}
\right.
  \end{split}
      \end{equation*}
\end{proof}

\begin{lem}\label{int1}
Let $0< \alpha\le 1$. If $f\in C^{L^\alpha Log^\nu L}(D)$ and $t\in \partial D$, then for $z, z+h\in \mathcal N(t)$ with 
$|h|\le \min\{h_0, \frac{\delta_0}{4}, \frac{s_0}{2}\}$,
\begin{equation*}
|Sf(z+h)-Sf(z)|\le \left\{
      \begin{array}{cc}
     C\|f\|_{C^{L^\alpha Log^\nu L}(D)}|h|^\alpha|\ln| h||^{\nu}, & 0<\alpha<1; \\
     C\|f\|_{C^{L^1 Log^\nu L}(D)}|h||\ln| h||^{\nu+1}, & \alpha =1
    \end{array}
\right.
\end{equation*}
for a constant $C$ dependent only on $D, \alpha$ and $\nu$.
\end{lem}
\begin{proof}Without loss of generality, assume $t\in \Gamma_1$. Since $z, z+h\in D$, by Cauchy integral formula, we have
\begin{equation*}
\begin{split}
Sf(z+h)-Sf(z) =& \frac{1}{2\pi i}\int_{\partial D}\frac{f(\zeta)-f(t)}{\zeta-z-h} -\frac{f(\zeta)-f(t)}{\zeta-z} d\zeta +\\
&+ \frac{f(t)}{2\pi i}(\int_{\partial D}\frac{1}{\zeta-z-h}d\zeta -\int_{\partial D} \frac{1}{\zeta-z} d\zeta)\\
=& \frac{1}{2\pi i}\int_{\partial D}\frac{f(\zeta)-f(t)}{\zeta-z-h} -\frac{f(\zeta)-f(t)}{\zeta-z} d\zeta\\
=& \frac{h}{2\pi i}\int_{\partial D}\frac{f(\zeta)-f(t)}{(\zeta-z-h)(\zeta-z)} d\zeta\\
=& \frac{h}{2\pi i}\int_{l}\frac{f(\zeta)-f(t)}{(\zeta-z-h)(\zeta-z)} d\zeta +  \frac{h}{2\pi i}\int_{\Gamma_1\setminus l}\frac{f(\zeta)-f(t)}{(\zeta-z-h)(\zeta-z)} d\zeta +\\
&+\frac{h}{2\pi i}\int_{\cup_{j=2}^N\Gamma_j}\frac{f(\zeta)-f(t)}{(\zeta-z-h)(\zeta-z)} d\zeta\\
=&: I+II+III.
\end{split}
\end{equation*}
Here $l$ is the arc on $\Gamma_1$ centered at $t$  of total  arclength $2|h|$. Note when $\zeta\in \cup_{j=2}^N\Gamma_j$, $|\zeta-z|\ge |\zeta-t| -|t-z| \ge \frac{3\delta_0}{4}$ and $|\zeta-z-h|\ge |\zeta-t| -|t-z| -|h| \ge \frac{\delta_0}{2}$. As in the proof of Lemma \ref{bound}, we immediately obtain
\begin{equation*}
    |III|\le C|h|\|f\|_{C(D)}\le C\|f\|_{C^{L^\alpha Log^\nu L}(D)}|h|^\alpha|\ln| h||^{\nu}.
\end{equation*}

For the remaining two terms $I$ and $II$, without loss of generality assume $|z-t|\ge |z+h-t|$. Then \begin{equation*}\label{333}
    |z-t|\ge \frac{1}{2}(|z-t|+|z+h-t|)\ge \frac{|h|}{2}.
\end{equation*}
Together with   (\ref{444}), we have \begin{equation}\label{111}
   |\zeta-z|\ge\max\{C|z-t|, C|\zeta-t|\} \ge \max\{C|h|, C|\zeta-t|\}.
\end{equation}
Recalling
\begin{equation*}\label{222}
    |\zeta-z-h|\ge \max\{C|z+h-t|, C|\zeta-t|\}\ge C|\zeta-t|, \end{equation*}
and combining it with (\ref{111}) and Lemma  \ref{elem}, one obtains
\begin{equation*}
\begin{split}
|I|
\le & C\|f\|_{C^{L^\alpha Log^\nu L}(D)}\int_l|\zeta-t|^{\alpha-1}|\ln|\zeta-t||^{\nu}|d\zeta|\\
\le & C\|f\|_{C^{L^\alpha Log^\nu L}(D)}\int_{0}^{|h|}s^{\alpha-1}|\ln s|^{\nu} ds\\
\le & C\|f\|_{C^{L^\alpha Log^\nu L}(D)} |h|^\alpha|\ln |h||^{\nu}.
\end{split}
\end{equation*}
Denote by $l'$ the arc on $\Gamma_1$ centered at $t$  of total arclength $\min\{2h_0, s_1\}$. Then
\begin{equation*}
\begin{split}
|II|\le &    C|h|\|f\|_{C^{L^\alpha Log^\nu L}(D)}\int_{l'\setminus l}\frac{|\zeta-t|^\alpha|\ln|\zeta-t||^{\nu}}{|\zeta-t|^2}|d\zeta| + C|h|\|f\|_{C(D)}\int_{\Gamma_1\setminus l'} \frac{1}{|\zeta-t|^2}|d\zeta|\\
\le & C|h|\|f\|_{C^{L^\alpha Log^\nu L}(D)}(\int_{|h|}^{h_0}s^{\alpha-2}|\ln s|^{\nu} ds+\int_{\min\{h_0, \frac{s_1}{2}\}}^{\frac{s_1}{2}}  \frac{1}{s^2}ds)\\
\le & \left\{
      \begin{array}{cc}
     C\|f\|_{C^{L^\alpha Log^\nu L}(D)}|h|^\alpha|\ln| h||^{\nu}, & 0<\alpha<1; \\
     C\|f\|_{C^{L^1 Log^\nu L}(D)}|h||\ln| h||^{\nu+1}, & \alpha =1.
    \end{array}
\right.
  \end{split}
\end{equation*}
\end{proof}

We now are in a position to estimate the Log-H\"older semi-norm of $Sf$ in $D$.

\begin{prop}\label{S1}Let $0<\alpha\le 1$. If $f\in C^{L^\alpha Log^\nu L}(D)$, then for $z, z+h\in D$ with $|h|\le \min\{\frac{h_0}{9c_0}, \frac{s_0}{18c_0}, \frac{\delta_0}{16}, \frac{e^{-\nu-1}}{3}\}$,
\begin{equation}\label{prop}
|Sf(z+h)- Sf(z)|\le \left\{
      \begin{array}{cc}
     C\|f\|_{C^{L^\alpha Log^\nu L}(D)}|h|^\alpha|\ln| h||^{\nu}, & 0<\alpha<1; \\
     C\|f\|_{C^{L^1 Log^\nu L}(D)}|h||\ln| h||^{\nu+1}, & \alpha =1
    \end{array}
\right.
\end{equation}
for a constant $C$ dependent only on $D, \alpha$ and $\nu$.
\end{prop}

\begin{proof}
 Let $t, t'\in\partial D$ such that $|z-t|=\text{dist}(z, \partial D)$ and $|z+h-t'| = \text{dist}(z+h, \partial D)$.  Without loss of generality, assume $t\in \Gamma_1$. If both $|z-t|$ and $|z+h-t'|$ are greater than $\frac{\delta_0}{16}$, then $|\zeta-z|\ge |z-t| \ge \frac{\delta_0}{16}$ and $|\zeta-z-h|\ge |t'-z-h|\ge\frac{\delta_0}{16} $ on $\zeta\in \partial D$. Consequently,
 \begin{equation*}
 \begin{split}
     |Sf(z+h)-Sf(z)| =&\left | \frac{h}{2\pi}\int_{\partial D}\frac{f(\zeta)-f(t)}{(\zeta-z-h)(\zeta-z)} d\zeta\right |\\
     \le &C|h|\|f\|_{C(D)}\\
     \le &C\|f\|_{C^{L^\alpha Log^\nu L}(D)}|h|^\alpha|\ln| h||^{\nu}.
     \end{split}
 \end{equation*}

 Otherwise, suppose one of $|z-t|$ and $|z+h-t'|$ is less  than $\frac{\delta_0}{16}$. Say $|z-t|\le \frac{\delta_0}{16}$, implying $|z+h-t'|\le |z+h-t|\le |z-t|+|h|\le \frac{\delta_0}{8}$.
  The other case is done similarly. Hence   $z\in \mathcal N(t)$ and $z+h\in \mathcal N(t')$ by definition.   
  Thus if in addition either $z+h\in \mathcal N(t)$ or $z\in \mathcal N(t')$, (\ref{prop}) follows directly from Lemma \ref{int1}.

 We are only left with the case  when both $z+h\in D \setminus \mathcal N(t)$ and $z\in D\setminus \mathcal N(t')$. Noticing that $|z+h-t|\le |z-t|+|h|< \frac{\delta_0}{4}$ and $|z-t'|\le |z-(z+h)|+|z+h-t'|< \frac{\delta_0}{4}$, it implies  by definition of $\mathcal N(t)$ and $\mathcal N(t')$ that  $ |z+h-t|\ge 4 |z+h-t'|$ and $  |z-t'|\ge 4|z-t|, $
 or equivalently,
 $$  |z+h-t'|\le \frac{1}{4}|z+h-t| \ \ \ \text{and}\ \  \ \  |z-t|\le \frac{1}{4}|z-t|. $$
  We claim that
 \begin{equation}\label{bd}
 |z+h-t'|\le |h|, |z-t|\le |h|,\ \text{and}\  |t-t'|\le 3|h|.\end{equation}
 Indeed, since $ |z+h-t'|\le \frac{1}{4}|z+h-t|\le \frac{1}{4}(|z+h-t'|+|t'-t|),$ we have
\begin{equation*}
   |z+h-t'|\le \frac{1}{3}|t'-t|.
\end{equation*}
Similarly, $$|z-t|\le \frac{1}{3}|t'-t|.$$
On the other hand, since
$|t'-t|\le |t'-z-h|+|z+h-z|+|z-t|\le \frac{2}{3}|t'-t|+|h|$, one infers
$$|t'-t|\le 3|h|.$$ Hence
$$|z+h-t'|\le |h|, \ \ |z-t|\le |h|.$$ The claim is proved.

Now we estimate $$ |Sf(z+h)-Sf(z)|   \le |Sf(z+h)-\Phi f(t')|+|Sf(z)-\Phi f(t)|+|\Phi f(t) -\Phi f(t')|$$ for  $z, z+h, t$ and $t'$ as previously.
Because $z+h\in \mathcal N(t')$ and  $|z+h-t'|\le |h|\le \min\{h_0, \frac{s_0}{2}, e^{-\nu-1}\}$ by (\ref{bd}), we deduce from Lemma \ref{bound},
\begin{equation*}
\begin{split}
    |Sf(z+h)-\Phi f(t')|&\le  \left\{
      \begin{array}{cc}& C\|f\|_{C^{ L^\alpha Log^\nu L}(D)} |z+h-t'|^\alpha|\ln |z+h-t'||^{\nu}, 0<\alpha<1 \\
     & C\|f\|_{C^{ L^1 Log^\nu L}(D)} |z+h-t'||\ln |z+h-t'||^{\nu+1}, \alpha=1
      \end{array}
\right.\\
&\le  \left\{
      \begin{array}{cc}& C\|f\|_{C^{ L^\alpha Log^\nu L}(D)} |h|^\alpha|\ln |h||^{\nu}, 0<\alpha<1; \\
     & C\|f\|_{C^{ L^1 Log^\nu L}(D)} |h||\ln |h||^{\nu+1}, \alpha=1.
      \end{array}
\right.
\end{split}
\end{equation*}
Here we have used the non-decreasing property of the real-valued functions $s^\alpha |\ln s|^{\nu}$ and $s |\ln s|^{\nu+1}$  when $s$ is less than $\min\{h_0, e^{-\nu-1}\}$. Similarly,
\begin{equation*}
   |Sf(z)-\Phi f(t)||\le  \left\{
      \begin{array}{cc}& C\|f\|_{C^{ L^\alpha Log^\nu L}(D)} |h|^\alpha|\ln |h||^{\nu}, 0<\alpha<1; \\
     & C\|f\|_{C^{ L^1 Log^\nu L}(D)} |h||\ln |h||^{\nu+1}, \alpha=1.
      \end{array}
\right.
\end{equation*}
Lastly, since $|t'-t|\le 3|h|\le  \min \{\frac{h_0}{3c_0}, \frac{s_0}{6c_0}, \frac{\delta_0}{2}, e^{-\nu-1}\}$, by Lemma \ref{int},
\begin{equation*}
\begin{split}
   |\Phi f(t)-\Phi f(t')||&\le  \left\{
      \begin{array}{cc}& C\|f\|_{C^{ L^\alpha Log^\nu L}(D)} |t-t'|^\alpha|\ln |t-t'||^{\nu}, 0<\alpha<1 \\
     & C\|f\|_{C^{ L^1 Log^\nu L}(D)} |t-t'||\ln |t-t'||^{\nu+1}, \alpha=1
      \end{array}
\right.\\
& \le \left\{
      \begin{array}{cc}& C\|f\|_{C^{ L^\alpha Log^\nu L}(D)} |h|^\alpha|\ln |h||^{\nu}, 0<\alpha<1; \\
     & C\|f\|_{C^{ L^1 Log^\nu L}(D)} |h||\ln |h||^{\nu+1}, \alpha=1.
      \end{array}
\right.
\end{split}
\end{equation*}
 The proof of the  proposition is  complete.
\end{proof}
\medskip

\begin{thm}\label{ok1}
 Let $D$ be a bounded domain in $\mathbb C$ with $C^{1,\alpha}$  boundary, $k\in \mathbb Z^+\cup \{0\}, 0<\alpha\le 1$. Then $S$ defined in (\ref{TS1}) sends  $C^{L^\alpha Log^\nu L}(D)$ into itself when $0<\alpha<1$, and into $ C^{L^1 Log^{\nu+1}  L}(D)$ if $\alpha=1$. Moreover, there exists a constant $C$ dependent only on $D, \alpha$ and $ \nu$, such that
 for any $f\in C^{L^\alpha Log^\nu L}(D)$,
$$
\|Sf\|_{C^{L^\alpha Log^\nu L}(D)}\le
     C\|f\|_{C^{L^\alpha Log^\nu L}(D)}$$ if $0<\alpha<1$,
and     $$ \|Sf\|_{C^{L^1 Log^{\nu+1}  L}(D)}\le C\|f\|_{C^{L^1 Log^{\nu}L}(D)}$$ if $\alpha =1$.
   \end{thm}
\begin{proof}
Choose $\epsilon$ such that $0<\epsilon< \alpha\le 1$. We have   $\|f\|_{C^{\epsilon}(D)}\le C\|f\|_{C^{ L^\alpha Log^\nu L}(D)}$ with $C$ dependent only on $\nu, \alpha, \epsilon$ and  $D$. Hence
$$\|Sf\|_{C(D)} \le \|Sf\|_{C^{\epsilon}(D)}\le C \|f\|_{C^{\epsilon}(D)}\le C\|f\|_{C^{L^\alpha Log^\nu L}(D)}.$$
The rest of the theorem  follows directly from Proposition \ref{S1}.
\end{proof}

\section{Proof of  Theorem \ref{main}}
We are now in a position to prove Theorem \ref{main}. Let $\Omega = D\times \Lambda\subset \mathbb C^2$, where $D\subset \mathbb C$ is a bounded domain with $C^{k+1, \alpha}$ boundary, and $\Lambda$ is an open set in $\mathbb R$ or $\mathbb C$. Let $S$ be defined in (\ref{S2}). For $0<\epsilon< \alpha\le 1$, 
there exists a constant $C$ dependent only on $\nu, \alpha, \epsilon$ and  $\Omega$, such that for all  $f\in C^{k, L^\alpha Log^\nu L}(\Omega)$, $$\|Sf\|_{C^k(\Omega)}\le C\|Sf\|_{C^{k, \epsilon}(\Omega)}\le C\|f\|_{C^{k, \epsilon}(\Omega)}\le C\|f\|_{C^{k, L^\alpha Log^\nu L}(\Omega)}.$$

We shall further prove for $|\gamma|=k$,
$
    H^{\nu+1}[D^\gamma Sf]\le C\|f\|_{C^{k, L^\alpha Log^\nu L}(\Omega)}.
$  Noticing that $S f$ is holomorphic with respect to $z\in D$, we  assume $D^\gamma =\partial_z^{\gamma_1}D_\lambda^{\gamma_2}$. Making use of integration by part, we obtain for any $(z, \lambda)\in \Omega$,
  \begin{equation*}
  \begin{split}
   D^\gamma S f(z, \lambda) =&\frac{1}{2\pi i}\partial_z^{\gamma_1} S D^{\gamma_2}_\lambda f(z, \lambda) \\
   = &\frac{1}{2\pi i}\partial_z^{\gamma_1-1}\int_{\partial D}\partial_{z} \frac{D^{\gamma_2}_\lambda f(\zeta, \lambda)}{\zeta-z}d\zeta \\
    =&\frac{1}{2\pi i}\partial_z^{\gamma_1-1}\int_{\partial D} \frac{\partial_{\zeta}D^{\gamma_2}_\lambda f(\zeta, \lambda)}{\zeta-z}d\zeta \\
    &\cdots\\
    =&: \frac{1}{2\pi i}\int_{\partial D} \frac{\tilde f(\zeta, \lambda)}{\zeta-z}d\zeta = S\tilde f(z, \lambda)
      \end{split}
  \end{equation*}
   with $\tilde f: = \partial_z^{\gamma_1}  D^{\gamma_2}_\lambda f\in C^{L^\alpha Log^\nu L}(\Omega)$ and $\|\tilde f\|_{C^{L^\alpha Log^\nu L}(\Omega)}\le \|f\|_{C^{k, L^\alpha Log^\nu L}(\Omega)}$. (See  \cite{PZ} Proposition 3.3, or  \cite{V} p. 21-22 for more details.) Therefore, it will suffice to show
  $H^{\nu+1}[S\tilde f]\le C\|\tilde f\|_{C^{ L^\alpha Log^\nu L}(\Omega)}$. By (the proof of) Proposition \ref{S1}, it is already clear that for each $\lambda\in \Lambda$,
  $S\tilde f(\zeta, \lambda)$  as a function of $\zeta\in D$ satisfies  $$H_D^{\nu+1}[S\tilde f(\cdot, \lambda)]\le C\|\tilde f\|_{C^{ L^\alpha Log^{\nu}L}(\Omega)}$$
  for a constant $C$ independent of $\tilde f$ and $\lambda$. In view of Lemma \ref{comp}, we only need to show  for each $z\in D$,  $S\tilde f(z, \zeta)$  as a function of $\zeta\in \Lambda$ satisfies \begin{equation}\label{rest}
    H_\Lambda^{\nu+1}[S\tilde f(z, \cdot)]\le C\|\tilde f\|_{C^{ L^\alpha Log^{\nu}L}(\Omega)}
\end{equation}
 for a constant $C$ independent of $\tilde f$ and $z$.

 To do so we shall apply the Maximum Modulus Principle of holomorphic functions.   First consider $z=t\in \partial D$. Without loss of generality, assume $t\in \Gamma_1$. By Sokhotski--Plemelj Formula, the non-tangential limit of $S\tilde f$ at $(t, \lambda)\in \partial D\times \Lambda$ is
 $$\Phi\tilde f(t, \lambda): = \frac{1}{2\pi i}\int_{\partial D}\frac{\tilde f(\zeta, \lambda)}{\zeta-t}d\zeta + \frac{1}{2}\tilde f(t, \lambda).$$
 Here the first term is interpreted as the Principal Value. We shall prove that for $\lambda, \lambda+h\in \Lambda$ with $0<|h|\le \min\{h_0, \frac{s_1}{2}\}$,
 \begin{equation}\label{S}
   \left |\int_{\partial D}\frac{\tilde f(\zeta, \lambda+h)}{\zeta-t}d\zeta - \int_{\partial D}\frac{\tilde f(\zeta, \lambda)}{\zeta-t}d\zeta\right|\le C |h|^{\alpha}|\ln |h||^{\nu+1} \|\tilde f\|_{C^{ L^\alpha Log^{\nu}L}(\Omega)}
 \end{equation}
 for a constant C independent of $\tilde f, t, \lambda$ and $h$.

 Indeed, write
 \begin{equation*}
 \begin{split}
  \int_{\partial D}\frac{\tilde f(\zeta, \lambda+h)-\tilde f(\zeta, \lambda)}{\zeta-t}d\zeta =& \int_{\partial D}\frac{\tilde f(\zeta, \lambda+h)-\tilde f(t, \lambda+h) -\tilde f(\zeta, \lambda)+ \tilde f(t, \lambda)}{\zeta-t}d\zeta\\
   &+ (\tilde f(t, \lambda+h)- \tilde f(t, \lambda))\int_{\partial D}\frac{1}{\zeta-t}d\zeta\\
    =&: I+II.
 \end{split}
 \end{equation*}
 Since $|\int_{\partial D}\frac{1}{\zeta-t}d\zeta|$ is bounded in terms of the Principal Value, $$|II|\le C|h|^{\alpha}|\ln |h||^{\nu+1} \|\tilde f\|_{C^{ L^\alpha Log^{\nu}L}(\Omega)}$$ for a constant $C$ independent of $\tilde f, t, \lambda$ and $h$.

 For $I$, let $l$ be the arc on $\partial D$ that is centered at $t$ with total arclength $2|h|$ and  $s$ be an arclength parameter of $\partial D$ such that  $\zeta|_{s=0}=t$. In particular, $l\subset \Gamma_1$. Then
 \begin{equation*}
 \begin{split}
   I = & \int_{l}\frac{\tilde f(\zeta, \lambda+h)-\tilde f(t, \lambda+h) -\tilde f(\zeta, \lambda)+ \tilde f(t, \lambda)}{\zeta-t}d\zeta+\\
   &+\int_{\Gamma_1\setminus l}\frac{(\tilde f(\zeta, \lambda+h)-\tilde f(t, \lambda+h)) -(\tilde f(\zeta, \lambda)-\tilde f(t, \lambda))}{\zeta-t}d\zeta+\\
   &+\int_{\cup_{j=2}^N \Gamma_j}\frac{(\tilde f(\zeta, \lambda+h)-\tilde f(\zeta, \lambda))-(\tilde f(t, \lambda+h) -\tilde f(t, \lambda))}{\zeta-t}d\zeta\\
   =&: I_1+I_2+ I_3.
   \end{split}
 \end{equation*}

Because $|\zeta-t|\ge \delta_0$ for $\zeta\in \cup_{j=2}^N \Gamma_j$ and $|\tilde f(\zeta, \lambda+h)-\tilde f(\zeta, \lambda))-(\tilde f(t, \lambda+h) -\tilde f(t, \lambda)|\le |h|^\alpha|\ln|h||^{\nu}\|\tilde f\|_{C^{ L^\alpha Log^{\nu}L}(\Omega)}$, one has
$$|I_3|\le C|h|^\alpha|\ln|h||^{\nu}\|\tilde f\|_{C^{ L^\alpha Log^{\nu}L}(\Omega)}$$
for a constant $C$ independent of $\tilde f, t, \lambda$ and $h$.

Recall by the chord-arc condition, $|\zeta-t|\approx |\zeta, t|=\min\{s, s_1-s\}$ on $\Gamma_1$. Moreover, the numerator of $I_1$ is less than $C|\zeta-t|^\alpha|\ln|\zeta-t||^{\nu} \|\tilde f\|_{C^{ L^\alpha Log^{\nu}L}(\Omega)}$. It follows from Lemma \ref{elem}
  $$|I_1|\le C\int_0^{|h|} s^{\alpha-1}|\ln s|^\nu ds \le C|h|^\alpha|\ln|h||^{\nu}\|\tilde f\|_{C^{ L^\alpha Log^{\nu}L}(\Omega)}$$ for a constant $C$ independent of $\tilde f, t, \lambda$ and $h$.

 Rearrange $I_2$ and we obtain
 \begin{equation*}
 \begin{split}
  | I_2| &\le \left |\int_{\Gamma_1\setminus l}\frac{\tilde f(\zeta, \lambda+h) - \tilde f(\zeta, \lambda)}{\zeta-t}d\zeta \right|+  \left |(\tilde f(t, \lambda+h)-\tilde f(t, \lambda))\int_{\Gamma_1\setminus l}\frac{1}{\zeta-t}d\zeta\right|\\
  & \le C|h|^\alpha|\ln |h||^{\nu}\|\tilde f\|_{C^{\alpha}(\Omega)}\int_{|h|}^{\frac{s_1}{2}} \frac{1}{s} ds +C|h|^\alpha|\ln|h||^{\nu}\|\tilde f\|_{C^{ L^\alpha Log^{\nu}L}(\Omega)}\\
   & \le  C|h|^{\alpha}|\ln|h||^{\nu+1}\|\tilde f\|_{C^{ L^\alpha Log^{\nu}L}(\Omega)}.
  \end{split}
 \end{equation*}

We have thus shown (\ref{S}) holds, and hence there exists a constant $C$  such that for each $z=t\in \partial D$, $H_\Lambda^{\nu+1}[\Phi \tilde  f(t, \cdot)]\le C\|\tilde f\|_{C^{L^\alpha Log^{\nu+1}  L}(\Omega)}$ with $C$ independent of $\tilde f$ and $t$. Notice that for each fixed $\zeta\in \Lambda$, $S \tilde f(z, \zeta)$ is holomorphic as a function of $z\in D$ and  by Plemelj--Privalov Theorem, continuous up to the boundary  with boundary value  $\Phi\tilde f(z, \zeta)$.  Applying the Maximum Modulus Theorem to the holomorphic function $\frac{S  \tilde f(z, \lambda+h) -S  \tilde f(z, \lambda)}{|h|^{\alpha}|\ln|h||^{\nu+1}}$ of $z\in D$ for each fixed $\lambda$ and $\lambda+h$ with $0<|h|\le \min\{h_0, \frac{s_0}{2}\}$, we deduce
\begin{equation*}
  \begin{split}
   \sup_{z\in D}\frac{|S  \tilde f(z, \lambda+h) -S  \tilde f(z, \lambda)|}{|h|^{\alpha}|\ln|h||^{\nu+1}}    \le &\sup_{t\in\partial D}\frac{|\Phi  \tilde f(t, \lambda+h) -\Phi  \tilde f(t, \lambda)|}{|h|^{\alpha}|\ln|h||^{\nu+1}}\\
   =  &\sup_{t\in\partial D}H_\Lambda^{\nu+1}[\Phi \tilde  f(t, \cdot)]\\
   \le & C\|\tilde f\|_{C^{L^\alpha Log^{\nu}L}(\Omega)},
  \end{split}
\end{equation*}
with $C$ independent of $\tilde f, z_1, z_2$ and $z'_2$. (\ref{rest}) is thus verified and the proof of Theorem \ref{main} is complete.
\medskip

We conclude the section by pointing out that   the proof of Tumanov's example in Section 3 indicates that  for any $\mu<1$, $S$ does not send  $C^{\alpha}  (\triangle^2)$ into $C^{L^\alpha Log^{\mu}  L}(\triangle^2), 0<\alpha<1$. Theorem \ref{main} thus is sharp in view of the example.



\section{Proof of Theorem \ref{cor} }
Let $D_j\subset\mathbb C$, $j= 1, \ldots, n$, be bounded domains with  $C^{k+1,\alpha}$ boundary, $n\ge 2$,  $k\in \mathbb Z^+\cup \{0\}, 0<\alpha\le 1$, and $\Omega: = D_1\times\cdots\times D_n$. Given a function $f\in C^{k, L^\alpha Log^\nu L} (\Omega)$, since $C^{k, L^\alpha Log^\nu L} (\Omega)\xhookrightarrow{i} C^{k,\epsilon} (\Omega)$ for $0<\epsilon<\alpha$, the following two operators are well defined for $z\in \Omega$,
\begin{equation}\label{TS}
  \begin{split}
    T_j f(z):&=-\frac{1}{2\pi i}\int_{D_j}\frac{f(z_1, \ldots, z_{j-1}, \zeta_j, z_{j+1}, \ldots, z_n)}{\zeta_j-z_j}d\bar\zeta_j\wedge\zeta_j;\\
    S_j f(z):&=\frac{1}{2\pi i}\int_{\partial D_j}\frac{f(z_1, \ldots, z_{j-1}, \zeta_j, z_{j+1}, \ldots, z_n)}{\zeta_j-z_j}d\zeta_j.
  \end{split}
\end{equation}

By Theorem \ref{main}, $S_j$ is a bounded operator sending $C^{k, L^\alpha Log^\nu L} (\Omega)$ into $C^{k, L^\alpha Log^{\nu+1}  L} (\Omega)$. It was proved in \cite{PZ} that the operator $T_j$ is  bounded  between  $C^{k, \alpha}(\Omega)$.   In the following, we generalize this result and show $T_j$ is bounded sending $C^{k, L^\alpha Log^\nu L} (\Omega)$ into itself.

\begin{prop}\label{ok}
For each $j\in\{1, \ldots, n\}$, $T_j$ is a bounded  operator sending $C^{k, L^\alpha Log^\nu L} (\Omega)$ into $C^{k, L^\alpha Log^\nu L} (\Omega)$, $k\in \mathbb Z^+\cup \{0\}, 0<\alpha\le 1,  \nu\in\mathbb R$. Namely, there exists a constant $C$ dependent only on $\Omega, k, \alpha$ and $\nu$, such that for $f\in C^{k, L^\alpha Log^\nu L} (\Omega)$,
      \begin{equation*}
    \begin{split}
       \|T_j f\|_{C^{k, L^\alpha Log^\nu L} (\Omega)}\le C\|f\|_{C^{k, L^\alpha Log^\nu L} (\Omega)}.
         \end{split}
  \end{equation*}
 \end{prop}

\begin{proof}
Without loss of generality, we assume $n=2$ and $j=1$. As in \cite{PZ},  $\|T_1f\|_{C^k(\Omega)}\le C\|f\|_{C^k(\Omega)}$ for a constant $C$ independent of $f$. 
We only need to show  $$H^{\nu}[D^\gamma T_1f]\le C\|f\|_{C^{k, L^\alpha Log^\nu L}(\Omega)}$$ for some constant independent of $f$ for all $|\gamma|=k$.

Write $D^\gamma = D^{\gamma_1}_1D^{\gamma_2}_2$. Then $D^\gamma T_1 f = D^{\gamma_1}_1 T_1(D^{\gamma_2}_2 f)$. If $\alpha<1$, choose a positive number $0<\epsilon<1- \alpha$. So $\alpha+\epsilon<1$ and for each $z_2\in D_2$, $\|D^\gamma T_1 f(\cdot, z_2)\|_{C^{L^\alpha Log^\nu L}(D_1)}\le C\|D^{\gamma_1}_1 T_1(D^{\gamma_2}_2 f)(\cdot, z_2)\|_{C^{\alpha+\epsilon}(D_1)}$  for some constant $C$ independent of $f$ and $z_2$.  We shall show for each $z_2\in D_2$, $ \|D^{\gamma_1}_1 T_1(D^{\gamma_2}_2 f)(\cdot, z_2)\|_{C^{\alpha+\epsilon}(D_1)} \le C\| f\|_{C^{|\gamma|}(\Omega)}$.  Indeed,  by making use of  Theorem \ref{123}, if $\gamma_1 =0$, $$\|D^{\gamma_1}_1 T_1(D^{\gamma_2}_2 f)(\cdot, z_2)\|_{C^{\alpha+\epsilon}(D_1)}  =\| T_1(D^{\gamma_2}_2 f)(\cdot, z_2)\|_{C^{\alpha+\epsilon}(D_1)}  \le C \|D^{\gamma_2}_2 f\|_{C(\Omega)}\le C\| f\|_{C^{\gamma_2}(\Omega)};$$ If $\gamma_1\ge 1$, then $$\|D^{\gamma_1}_1 T_1(D^{\gamma_2}_2 f)(\cdot, z_2)\|_{C^{\alpha+\epsilon}(D_1)}\le  C\|D^{\gamma_2}_2 f\|_{C^{\gamma_1-1, \alpha+\epsilon}(\Omega)}\le C\| f\|_{C^{\gamma_1+\gamma_2}(\Omega)}$$
for some constant $C$  independent of $f$ and $z_2$. Altogether,  $D^\gamma T_1 f(\zeta, z_2)$ as a function of $\zeta\in D_1$ satisfies $$\|D^\gamma T_1 f(\cdot, z_2)\|_{C^{L^\alpha Log^\nu L}(D_1)}\le C\|D^{\gamma_1}_1 T_1(D^{\gamma_2}_2 f)(\cdot, z_2)\|_{C^{\alpha+\epsilon}(D_1)}\le C\| f\|_{C^{|\gamma|}(\Omega)}\le C\|f\|_{C^{k, L^\alpha Log^\nu L}(\Omega)}$$
for some constant $C$  independent of $f$ and $z_2$.
If $\alpha=1\ ( \text{ so} \ \nu\ge 0)$, choose $\epsilon<1$. Then $ \|D^\gamma T_1 f(\cdot, z_2)\|_{C^{L^1 Log^\nu L}(D_1)}\le C\|D^{\gamma_1}_1 T_1(D^{\gamma_2}_2 f)(\cdot, z_2)\|_{C^{1}(D_1)}\le C\| T_1(D^{\gamma_2}_2 f(\cdot, z_2))\|_{C^{\gamma_1+1, \epsilon}(D_1)}$ and hence by Theorem \ref{123}, $$\|D^\gamma T_1 f(\cdot, z_2)\|_{C^{L^1 Log^\nu L}(D_1)}
\le C\|D^{\gamma_2}_2 f\|_{C^{\gamma_1, \epsilon}(\Omega)}\le C\|f\|_{C^{|\gamma|, \epsilon}(\Omega)}\le C\|f\|_{C^{k, L^1 Log^\nu L}(\Omega)}$$ for some $C$ independent of $f$ and $z_2$.

 Let $z'_2(\ne z_2)\in D_2$ with $|z_2-z_2'|\le h_0$ and consider $F_{z_2, z'_2}(\zeta): = \frac{D^{\gamma_2}_2 f(\zeta, z_2)-D^{\gamma_2}_2 f(\zeta, z'_2)}{|z_2-z'_2|^\alpha|\ln |z_2-z_2'||^{\nu}}$ on $D_1$. Since $f\in C^{k,L^\alpha Log^\nu L}(\Omega)$, $F_{z_2, z'_2}\in C^{\gamma_1}(D_1)$ and $\| F_{z_2, z'_2}\|_{C^{\gamma_1}(D_1)}\le \|f\|_{C^{k,L^\alpha Log^\nu L}}(\Omega)$. If $\gamma_1=0$, $$\|D^{\gamma_1}_1T_1 F_{z_2, z'_2}\|_{C(D_1)} =\| T_1 F_{z_2, z'_2}\|_{C(D_1)} \le C \|F_{z_2, z'_2}\|_{ C(D_1)}  \le C\|f\|_{C^{k,L^\alpha Log^\nu L}(\Omega)}$$
for some constant $C$ independent of $f$, $z_2$ and $ z'_2$.  For $\gamma_1\ge 1$, choosing $\epsilon<\alpha$,
  we have from Theorem \ref{123},  $$\|D^{\gamma_1}_1T_1 F_{z_2, z'_2}\|_{C(D_1)}\le C \|F_{z_2, z'_2}\|_{ C^{\gamma_1-1, \epsilon}(D_1)} \le C\| F_{z_2, z'_2}\|_{C^{\gamma_1}(D_1)}  \le C\|f\|_{C^{k, L^\alpha Log^\nu L}(\Omega)}$$
    for some constant $C$ independent of $f$, $z_2$ and $ z'_2$. Hence for each $z_1\in D_1$,
  $$\frac{|D^\gamma T_1 f(z_1, z_2)-D^\gamma T_1f(z_1, z'_2)|}{|z_2-z'_2|^\alpha|\ln |z_2-z_2'||^\nu} = |D^{\gamma_1}_1 T_1 F_{z_2, z'_2}(z_1)|\le\|D^{\gamma_1}_1T_1 F_{z_2, z'_2}\|_{C(D_1)}\le C\|f\|_{C^{k, L^\alpha Log^\nu L}(\Omega)},$$
where $C$ is independent of $f$, $z_1, z_2$ and $z'_2$. The proof of the proposition is  complete in view of Lemma \ref{comp}.
\end{proof}

\medskip

\begin{thm}\label{ka}
 Let $\mathbf f=\sum_{j=1}^nf_jd\bar z_j\in C^{k, L^\alpha Log^{\nu}  L}(\Omega)$, $k\in \mathbb Z^+\cup\{0\},  0<\alpha\le 1$ and $ \nu\in \mathbb R$. Then
\begin{equation}\label{key}
  T\mathbf f: = \sum_{j=1}^n\prod_{l=1}^{j-1}T_jS_lf_j= T_1f_1+T_2S_1f_2+\cdots+ T_nS_1\cdots S_{n-1}f_n
\end{equation}
is in  $ C^{k, L^\alpha Log^{\nu +n-1}L}(\Omega)$   with $\|T\mathbf f\|_{C^{k, L^\alpha Log^{\nu +n-1}L}(\Omega)}\le C \|\mathbf f\|_{C^{k, L^\alpha Log^{\nu}L}(\Omega)}$ for some constant $C$ dependent only on $\Omega, k, \alpha$ and $\nu$.
\end{thm}

\begin{proof}
 The operator $T$ in (\ref{key}) is well defined on $C^{k,  L^\alpha Log^{\nu}L }(\Omega)$ due to  Theorem  \ref{ok1} and Proposition \ref{ok}.  Moreover, for each $1\le j\le n$,
  \begin{equation*}
    \begin{split}
      \|\prod_{l=1}^{j-1}T_jS_lf_j\|_{C^{k,  L^\alpha Log^{\nu +n-1}L}(\Omega)}&\le C \|\prod_{l=1}^{j-1}S_lf_j\|_{C^{k,  L^\alpha Log^{\nu +n-1}L}(\Omega)}\\
      &\le C \|\prod_{l=1}^{j-2}S_lf_j\|_{C^{k,  L^\alpha Log^{\nu +n-2}L}(\Omega)}\\
      &\le \cdots\\
      &\le C \|f_j\|_{C^{k, L^\alpha Log^{\nu+n-j}L}(\Omega)}\\
      &\le C \|f_j\|_{C^{k, L^\alpha Log^{\nu}L}(\Omega)}.
    \end{split}
  \end{equation*}
 Therefore, $\|T\mathbf f\|_{C^{k,  L^\alpha Log^{\nu +n-1}L}(\Omega)}\le C \|\mathbf f\|_{C^{k,  L^\alpha Log^{\nu}L}(\Omega)} $.
 \end{proof}

\begin{proof}[Proof of Theorem \ref{cor}]
 When $\mathbf f$ is $\bar\partial$-closed, $T\mathbf f$ defined by (\ref{key}) satisfies $\bar\partial T\mathbf f = \mathbf f$ (in the sense of distributions if $k=0$) by \cite{PZ}. The rest of the theorem follows from Theorem \ref{ka}.
\end{proof}

\begin{proof}[Proof of Example \ref{ex2}]
   $\mathbf f$ is well defined in $\triangle^2$ and $\mathbf f = (z_1-1)^{k+\alpha} \log^\nu (z_1-1)d\bar z_2\in C^{k, L^\alpha Log^\nu L}(\triangle^2)$. Assuming    $u\in C^{k, L^\beta Log^\nu L}(\triangle^2)$ solves $\bar\partial u =\mathbf f$ in $\triangle^2$ for some $\beta>\alpha$, then there exists a holomorphic function $h$  in $\triangle^2$ such that $u = h +(z_1-1)^{k+\alpha}\log^\nu  (z_1-1)\bar z_2$.

  Now  consider $w(\xi): =\int_{|z_2|=\frac{1}{2}}u(\xi, z_2)dz_2$ on $\xi\in \triangle=\{z\in \mathbb C: |z|<1\}$. Since $u\in C^{k, L^\beta Log^\nu L}(\triangle^2)$,  $w\in C^{k, L^\beta Log^\nu L}(\triangle)$ as well. On the other hand, a direct computation gives
  \begin{equation*}
    \begin{split}
     w(\xi) =&\int_{|z_2|=\frac{1}{2}} (\xi-1)^{k+\alpha}\log^\nu  (\xi-1)\bar z_2dz_2 \\
     =& (\xi-1)^{k+\alpha}\log^\nu  (\xi-1)\int_{|z_2|=\frac{1}{2}} \frac{1}{4 z_2}dz_2\\ =&\frac{\pi i(\xi-1)^{k+\alpha}\log^\nu  (\xi-1)}{2}.
   \end{split}
  \end{equation*}
  This  contradicts with the fact that $(\xi-1)^{k+\alpha}\log^\nu  (\xi-1)\notin C^{k, L^\beta Log^\nu L}(\triangle)$ for any $\beta>\alpha$.
  \end{proof}

 \noindent Yifei Pan, pan@pfw.edu, Department of Mathematical Sciences, Purdue University Fort Wayne,
Fort Wayne, IN 46805-1499, USA \medskip

\noindent Yuan Zhang, zhangyu@pfw.edu, Department of Mathematical Sciences, Purdue University Fort Wayne,
Fort Wayne, IN 46805-1499, USA


\begin{thebibliography}{10}












\bibitem{GT}{\sc Gilbarg. D.; Trudinger, N. S.}: {\em Elliptic partial differential equations of second order.} Classics in Mathematics. Springer-Verlag, Berlin, 2001. xiv+517 pp.








\bibitem{Kenig}{\sc Kenig, C. E.}: {\em Recent progress on boundary value problems on Lipschitz domains.} Proc. Sympos. Pure Math. \textbf{43} (1985), 175--205.


\bibitem{Kerzman}{\sc Kerzman, N.}: {\em H\"{o}lder and $L\sp{p}$ estimates for solutions of $\bar \partial u=f$ in strongly pseudoconvex domains.} Comm. Pure Appl. Math. \textbf{24}(1971) 301--379.



\bibitem{Mclean}{\sc Mclean, W.}: {\em H\"older estimates for the Cauchy integral on a Lipschitz contour.} J. Integral Equations Appl. \textbf{1} (1988), no. 3, 435--451.

\bibitem{MP}{\sc Merker, J.; Porten, E.}: {\em Holomorphic extension of CR functions, envelopes of holomorphy, and removable singularities.} IMRS Int. Math. Res. Surv. 2006, Art. ID 28925, 287 pp.

\bibitem{Muskhelishvili} {\sc Muskhelishvili, N. I.}: {\em Singular integral equations. Boundary problems of function theory and their application to mathematical physics}.  Dover Publications, Inc., New York, 1992. 447 pp.


\bibitem{Tumanov} {\sc Tumanov, A.}: {\em
On the propagation of extendibility of CR functions. } Complex analysis and geometry (Trento, 1993), 479--498,
Lecture Notes in Pure and Appl. Math., 173, Dekker, New York, 1996.


\bibitem{PZ} {\sc Pan, Y.; Zhang, Y.}: {\em $C^{k, \alpha}$ estimates of the $\bar\partial$ equation on product domains.} Preprint.


\bibitem{V} {\sc Vekua, I. N.}: {\em Generalized analytic functions},
  vol.~29, Pergamon Press Oxford, 1962.

\end{thebibliography}
\end{document}